%% file: torsion.tex
\documentclass{mcom-l}
\usepackage{booktabs}
\usepackage[usenames,dvipsnames]{xcolor}
\usepackage{hyperref}
\hypersetup{colorlinks=true,urlcolor=RoyalPurple,citecolor=blue,linkcolor=blue}
\usepackage{amssymb}
\usepackage{nicefrac}

\newcommand{\Fq}{\mathbb{F}_q}

\newcommand{\Z}{\mathbb{Z}}

\newtheorem{proposition}{Proposition}

\newtheorem{lemma}{Lemma}
\newtheorem{algorithm}{Algorithm}

\theoremstyle{definition}

\newtheorem{example}{Example}

\theoremstyle{remark}

\begin{document}

\title[Constructing elliptic curves with prescribed torsion]{Constructing elliptic curves over finite fields with prescribed torsion}
\author{Andrew V. Sutherland}
\address{Department of Mathematics, Massachusetts Institute of Technology, Cambridge, Massachusetts 02139}
\curraddr{}
\email{drew@math.mit.edu}
\date{}
\dedicatory{}
\copyrightinfo{}

\begin{abstract}
We present a method for constructing optimized equations for the modular curve $X_1(N)$ using a local search algorithm on a suitably defined graph of birationally equivalent plane curves.
We then apply these equations over a finite field $\Fq$ to efficiently generate elliptic curves with nontrivial $N$-torsion by searching for affine points on $X_1(N)(\Fq)$,
and we give a fast method for generating curves with (or without) a point of order $4N$ using $X_1(2N)$.
\end{abstract}

\maketitle

\section{Introduction}

By Mazur's theorem \cite{Mazur:Torsion}, the order of a nontrivial torsion point on an elliptic curve over the rational numbers must belong to the set
$$\mathcal{T} = \{2,3,4,5,6,7,8,9,10,12\}.$$
Conversely, for each $N\in \mathcal{T}$ an infinite family of elliptic curves over $\mathbb{Q}$ containing a point of order~$N$ is exhibited by the parameterizations of
Kubert \cite{Kubert:Torsion}.\footnote{Kubert also addresses the torsion subgroups $\mathbb{Z}/2\mathbb{Z}\times \mathbb{Z}/2N\mathbb{Z}$ for $N=1,2,3,4$.}
Over a finite field $\mathbb{F}_q$, these parameterizations provide an efficient way to generate curves whose order $\#E(\Fq)$ is divisible by $N$.
This can significantly accelerate applications that search for an elliptic curve $E/\Fq$ with a desired property, such as curves with smooth order, as in the elliptic
curve factorization method \cite{Atkin:ECMCurves,Montgomery:SpeedingECM}, or curves with a particular endomorphism ring, as when computing Hilbert class polynomials with the Chinese Remainder Theorem \cite{Belding:HilbertClassPolynomial,EngeSutherland:CRTClassInvariants,Sutherland:HilbertClassPolynomials}.

To generate an elliptic curve $E/\Fq$ with nontrivial 7-torsion, for example, we simply pick an element $r\in \mathbb{F}_q$ and use $b=r^3-r^2$ and $c=r^2-r$ to define
\begin{equation}\label{equation:KubertCurve}
E(b,c):\quad y^2 + (1-c)xy - by = x^3-bx^2.
\end{equation}
Provided $E(b,c)$ is nonsingular, we obtain an elliptic curve on which the point $P=(0,0)$ has order 7.  By contrast, obtaining such a curve by trial and error is far more time consuming:
testing for 7-torsion typically involves finding the roots of a degree-24 polynomial (the 7-division polynomial), and several curves may need to be tested (approximately six, on average) .

Mazur's theorem limits us to $N\in \mathcal{T}$, but we may proceed further if we do not restrict ourselves to curves defined over $\mathbb{Q}$.
Reichert treats $N\in \{11,13,14,15,16,18\}$ over quadratic extensions of $\mathbb{Q}$ in \cite{Reichert:NontrivialTorsion} by finding explicit
equations for the modular curve $X_1(N)$, whose non-cuspidal points parametrize elliptic curves with a distinguished point of order~$N$.
We may be able to reduce a curve defined over a quadratic field $\mathbb{Q}[\sqrt{d}]$ to $\mathbb{F}_q$, but only when $d$ is a quadratic residue.

Alternatively, we can use an $\Fq$-rational point on $Y_1(N)$, the affine part of $X_1(N)$, to directly construct a curve $E(b,c)/\Fq$ containing a point of order~$N$, for any sufficiently large $q$ prime to $N$.
For $N\in \mathcal{T}$ the curve $X_1(N)$ has genus 0 and we may use Kubert's parameterizations, but
in general, we construct $E(b,c)$ from a point on $Y_1(N)(\Fq)$ via a birational map that depends on our choice of an explicit equation for $Y_1(N)$.

For example, to construct an elliptic curve $E/\Fq$ with a point of order~$13$, we start by finding an $\Fq$-rational point $(x,y)$ on the affine curve
\begin{equation}\label{equation:Y13}
y^2 + (x^3+x^2+1)y - x^2-x = 0,
\end{equation}
which may be done by choosing $x\in\mathbb{F}_q$ at random\footnote{For $N\in\{11,14,15\}$ the curve $X_1(N)$ has genus 1 and we may
obtain additional points more efficiently using the elliptic curve group operation, as discussed in Section \ref{section:FiniteFields}.} and attempting to solve the resulting quadratic equation for $y$ in $\mathbb{F}_q$.
We then apply the transformation
\begin{align*}
r &= 1 - xy,\\
s &= 1-xy / (y+1),
\end{align*}
set $c=s(r-1)$ and $b=cr$, and construct $E(b,c)$.
In the unlikely event that $E(b,c)$ is singular over $\Fq$, we simply look for a different point on the curve (\ref{equation:Y13}).

To apply this method we require a defining equation for $Y_1(N)$ and a corresponding birational map.
For fast computation we seek a plane curve $f(x,y)=0$ that minimizes the degree $d$ of one of its variables.
For $N\le 18$ one can derive these from the results of Kubert and Reichert.
Reichert's method can be applied to $N > 18$, but the resulting equation is quite large and of higher degree than necessary.
More compact defining equations are given by Yang \cite{Yang:ModularCurves} for $N\le 22$ and Baaziz \cite{Baaziz:ModularCurveX1N} for $N\le 51$, but these do not necessarily minimize~$d$, which critically impacts the efficiency of the construction above.

The minimal $d=d(N)$ corresponds to the \emph{gonality} of the curve $X_1(N)$, and is a topic of independent interest \cite{Jeon:TorsionQuarticNumberField,Jeon:TorsionCubicNumberField,Kamienny:UBC,Merel:UBC,Parent:TorsionNumberFields},
since it implies that there are infinitely many elliptic curves containing a point of order~$N$ defined over number fields of degree $d(N)$.
For $N>18$ only a few values of $d(N)$ are known,
% (see sequence A146879 in the OEIS \cite{Sloane:OEIS}),
but explicit equations provide upper bounds on $d(N)$, and can be used to define parametrized families of elliptic curves over number fields of a particular degree \cite{Jeon:ExplicitTorsionQuarticNumberField,Jeon:ExplicitTorsionCubicNumberField}.

Given a defining equation $f(x,y)=0$, one may attempt to reduce its complexity (degree, number of terms, and coefficient size) through a judiciously chosen sequence of rational transformations \cite{Baaziz:ModularCurveX1N,Reichert:NontrivialTorsion,Yang:ModularCurves}.
However this procedure tends to be rather \emph{ad hoc}, and becomes increasingly difficult as $N$ grows.
Here we treat this as a combinatorial optimization problem and apply standard search techniques to obtain a solution that is locally optimal under a metric we define.

Our results are not necessarily globally optimal, but for $N\le 22$ we are able to match known lower bounds for $d(N)$ \cite{Jeon:TorsionQuarticNumberField,Jeon:TorsionCubicNumberField,Kamienny:UBC}, and for $N\le 50$ we are able to match or improve the best known upper bounds for $d(N)$ given by explicit equations.
However, we do not achieve $d(24)=4$, implied by the (non-explicit) result in \cite{Jeon:TorsionQuarticNumberField}.
Optimized equations for $N\le 30$ appear in the appendix, and are available online for $N\le 50$.
The local search strategy we describe here has recently been augmented using simulated annealing~\cite{CadaySutherland:SimulatedAnnealing}, extending our results to $N\le 101$.

For odd $N$ we also show how to quickly generate $E/\mathbb{F}_q$ with a point of order~$4N$, or satisfying $\#E(\Fq)\equiv 2N\bmod 4N$, using our equations for $Y_1(2N)$.
These techniques play an important role in \cite{Sutherland:HilbertClassPolynomials}, and we expect they have other applications.

\section{Computing the ``raw form" of $X_1(N)$}\label{section:rawform}

Following \cite{Reichert:NontrivialTorsion}, we give a method to compute a defining equation $F_N(r,s)=0$ for $Y_1(N)$ that Reichert calls the ``raw form" of $X_1(N)$.\footnote{Reichert also uses auxiliary variables $m=s(1-r)/(1-s)$ and $t=(r-s)/(1-s)$.  We work directly with $r$ and $s$ throughout.}
The equation $E(b,c)$ is the Tate normal form of an elliptic curve (called a Kubert curve in \cite{Atkin:ECMCurves}).
Any elliptic curve with a point of order greater than $3$ can be put in this form \cite[\S V.5]{Knapp:EllipticCurves}.
The discriminant of $E(b,c)$ is
\begin{equation}\label{equation:discriminant}
\Delta(b,c) =b^3(16b^2 - 8bc^2 - 20bc+b + c(c-1)^3).
\end{equation}
To ensure that $E(b,c)$ is nonsingular we require $\Delta(b,c)\ne 0$, so we assume that $b\ne 0$.
Applying the group law for elliptic curves \cite[III.2.3]{Silverman:EllipticCurves1}, we double the point $P=(0,0)$ to obtain $2P=(b,bc)$, and for $n>1$ we may compute the point $(n+1)P=(x_{n+1},y_{n+1})$ in terms of $nP=(x_n,y_n)$ using
\begin{equation}\label{equation:grouplaw}
x_{n+1} = by_n/x_n^2, \quad y_{n+1} = b^2(x_n^2-y_n)/x_n^3.
\end{equation}
The inverse of the point $nP=(x_n,y_n)$ is
\begin{equation}\label{equation:inverse}
-nP = (x_n,b+(c-1)x_n-y_n).
\end{equation}
If $P$ is an $N$-torsion point and $m+n=N$, then we must have $mP=-nP$.  If $m\ne n$ this implies $x_m=x_n$, and if $m=n$ we have $2y_n = b+(c-1)x_n$.
When $x_m=x_n$, either $mP=nP$ or $mP=-nP$, and in the latter case $P$ is an $N$-torsion point.
If we choose $m=\left\lceil\frac{N+1}{2}\right\rceil$ and $n=\left\lfloor\frac{N-1}{2}\right\rfloor,$ we ensure $mP\ne nP$, obtaining a necessary and sufficient condition for $N$-torsion:
\begin{equation}\label{equation:torsion}
NP=0_E\quad\Longleftrightarrow\quad x_m=x_n,
\end{equation}
valid for all $N>3$, where $0_E$ denotes the neutral point on $E(b,c)$.
\smallskip

The first three multiples of $P$ are:
\begin{align*}
P &= (0,0),\\
2P &= (b,bc),\\
3P &= (c,b-c).
\end{align*}
None of these points is $0_E$, thus $P$ always has order greater than 3.
Applying (\ref{equation:torsion}), we see that $P$ is a point of order 4 if and only $c=0$, and $P$ is a point of order 5 if and only $b=c$.
For $N$ greater than $5$ we define:
\begin{align*}
r &= b/c,\qquad\qquad\qquad b = rs(r-1),\\
s &= c^2/(b-c),\qquad\quad\hspace{1pt} c = s(r-1),
\end{align*}
where $r\ne0,1$, and $s\ne 0$, since $b\ne 0$.
We may apply (\ref{equation:grouplaw}) to iteratively compute~$x_n$ as a rational function of $r$ and $s$; values for $1\le n\le 10$ are listed in Table \ref{table:Px}.

We now give an algorithm to compute $F_N(r,s)$ for $N > 5$, working in the polynomial ring $\Z[r,s]$.
We assume that the polynomials $F_M$ have already been computed, for $5 < M < N$, and that the rational function $x_n(r,s)$ is in the form $x_n=v_n/w_n$, where $v_n$ and $w_n$ are relatively prime polynomials in $\Z[r,s]$.

\begin{algorithm}
Given an integer $N>5$, compute $F_N(r,s)$ as follows:
\end{algorithm}
\renewcommand\labelenumi{\theenumi.}
\begin{enumerate}
\item
Compute $G_N = v_mw_n - v_nw_m$, where $m=\left\lceil\frac{N+1}{2}\right\rceil$ and $n=\left\lfloor\frac{N-1}{2}\right\rfloor$.
\item
Remove any powers of $r$, $s$,$(r-1)$, or $F_M$ that divide $G_N$,\\
for all $M>5$ properly dividing $N$.
\item
Make the remaining polynomial square-free and output the result as $F_N(r,s)$.
\end{enumerate}

\begin{example}
For $N=16$ we have $m=9$ and $n=7$.  After computing $x_9=v_9/w_9$ and $x_7=v_7/w_7$ (which may be found in Table \ref{table:Px}), step 1 of the algorithm constructs
the polynomial
\begin{align*}
G_{16}(r,s) =\medspace &s(r-1)(rs-2r+1)(rs^2-3rs+r+s+s^2) \cdot (r-s)^2\\
             &- rs(r-1)(s-1)(rs-2r+1) \cdot (r-s^2+s-1)^2.
\end{align*}
In step 2 the factors $s$, $r-1$, and $F_8(r,s)=rs-2r+1$ are removed, yielding the square-free polynomial
\begin{align*}
F_{16}(r,s) =\medspace &r^3s^2 - 4r^3s + 2r^3 + 3r^2s^2 + 2r^2s - 2r^2 - rs^5 + 4rs^4\\
					    &- 10rs^3 + 6rs^2 - 3rs + r + s^4,
\end{align*}
which appears in Table \ref{table:RawX1}.
\end{example}

For practical implementation it is convenient to first compute $x_n$ for all $n$ up to some bound $B$, and then use these to compute $F_N$ for $N\le 2B-1$.
The polynomials $F_N(r,s)$ for $N\le 101$ can be found at \url{http://math.mit.edu/~drew}.

\begin{proposition}
Let $F_N(r,s)$ be the polynomial output by Algorithm 1 on input $N > 5$.  Let $b=r_0s_0(r_0-1)$ and $c=s_0(r_0-1)$ with $\Delta(b,c)\ne 0$,
where $r_0$ and $s_0$ lie in a field whose characteristic does not divide $N$.
\smallskip

Then $P=(0,0)$ is a point of order $N$ on $E(b,c)$ if and only if $F_N(r_0,s_0)=0$.
\end{proposition}
\begin{proof}
We first note that $\Delta(b,c)\ne 0$ implies $b\ne 0$, hence $r_0\ne 0,1$ and $s_0\ne 0$.  We thus have $c\ne 0$ and $b\ne c$, so $P$ has order greater than 5.
If $w_m(r_0,s_0)=0$ in step 1 of the algorithm, then $P$ is an $m$-torsion point with $v_m(r_0,s_0)=b^2\ne0$, and we must have $w_n(r_0,s_0)\ne0$, since $m-n\le 2$ and $P$ has order greater than~$2$.
Similarly, if $w_n(r_0,s_0)=0$ then $v_n(r_0,s_0)$ and $w_m(r_0,s_0)$ are both nonzero.  It follows that $G_N(r_0,s_0)=0$ if and only if $x_m(r_0,s_0)=-x_n(r_0,s_0)$, equivalently,
if and only if $NP=0_E$.  This proves the proposition in that case that $N$ has no proper factors greater than 5, since $P$ has order greater than 5, and it proves the forward implication in every case.

For the reverse implication we proceed by induction on the number of proper factors of $N$ greater than~$5$, appealing to the moduli interpretation of $Y_1(N)$.
For $N>5$ (in fact, $N\ge 5$) the irreducible affine curve $Y_1(N)$ is a fine moduli space for equivalence classes of pairs $(E,P)$, where $P$ is a point of order~$N$ on the elliptic curve~$E$ (this is classical, see, e.g., \cite[\S 4]{Gross:TamenessCriterion} for a modern summary and references).
Each equivalence class is represented by a curve $E(b,c)$ with the point $P=(0,0)$ of order $N$, corresponding to a pair $(r_0,s_0)$.
There is a surjective rational map from the set $\{(r_0,s_0):G_N(r_0,s_0)=0, r_0\ne 0,1, s_0\ne 0\}$ to the union of the curves $Y_1(M)$ with $M>5$ dividing~$N$.
By the inductive hypothesis, the inverse image of $Y_1(M)$ under this map is the zero locus of $F_M$, for each $M>5$ properly dividing~$N$.  After these factors are removed in step 4,
the points on the remaining curve $F_N(r,s)=0$ must be mapped onto $Y_1(N)$.\footnote{This argument is essentially a formalization of the remarks in \cite[pp.~88--89]{Husemoller:EllipticCurves}.}
\end{proof}

%We note that the polynomials $F_N(r,s)$ have integer coefficients and are defined over any field.
%In the proposition above, the curve $E(b,c)$ is defined over any field containing $r_0$ and $s_0$.

\begin{table}
\begin{center}
\begin{tabular}{@{}l@{}}
\toprule
$x_1 = 0$\\
$x_2=rs(r-1)$\\
$x_3=s(r-1)$\\
$x_4=r(r-1)$\\
$x_5=rs(s-1)$\\
$x_6=s(r-1)(r-s)\enspace/\enspace(s-1)^2$\\
$x_7=rs(r-1)(s-1)(rs-2r+1)\enspace/\enspace(r-s)^2$\\
$x_8=r(r-1)(r-s)(r-s^2+s-1)\enspace/\enspace(rs-2r+1)^2$\\
$x_9=s(r-1)(rs-2r+1)(rs^2-3rs+r+s^2)\enspace/\enspace(r-s^2+s-1)^2$\\
$x_{10}=rs(r-s^2+s-1)(r^2-rs^3+3rs^2-4rs+s)\enspace/\enspace(rs^2-3rs+r+s^2)^2$\\
$x_{11}=rs(r-1)(s-1)(rs^2-3rs+r+s^2)(r^2s-3r^2+rs+3r-s^2-1)$\\
%$\qquad\quad /\enspace (r^2 - rs^3 + 3rs^2 - 4rs + s)^2$\\
%$x_{12}=(r-1)(r^2 - rs^3 + 3rs^2 - 4rs + s)$\\
%$\qquad\quad(r^3 - r^2s^4 + 5r^2s^3 - 9r^2s^2 + 4r^2s - 2r^2 - rs^3 + 6rs^2 - 3rs + r - s^3)$\\
%$\qquad\quad / \enspace ((s-1)^2(r^2s - 3r^2 + rs + 3r - s^2 - 1))^2$\\
%$x_{13}=rs(r-1)(s-1)(r-s)(r^2s - 3r^2 + rs + 3r - s^2 - 1)$\\
%$\qquad\quad(r^2s^3 - 5r^2s^2 + 6r^2s - r^2 + rs^4 - 3rs^3 + 6rs^2 - 7rs + r + s)$\\
%$\qquad\quad / \enspace (r^3 - r^2s^4 + 5r^2s^3 - 9r^2s^2 + 4r^2s - 2r^2 - rs^3 + 6rs^2 - 3rs + r - s^3)^2$\\
%$x_{14}=rs(r-1)(r^3 - r^2s^4 + 5r^2s^3 - 9r^2s^2 + 4r^2s - 2r^2 - rs^3 + 6rs^2 - 3rs + r - s^3)$\\
%$\qquad\quad(r^3 - r^2s^5 + 7r^2s^4 - 18r^2s^3 + 19r^2s^2 - 10r^2s - rs^5 + 4rs^4 - 5rs^2 + 5rs - s^5 + s^4$\\
%$\quad\qquad -\enspace s^3 + s^2 - s)$\\
%$\qquad\quad / \enspace ((r-s)(r^2s^3 - 5r^2s^2 + 6r^2s - r^2 + rs^4 - 3rs^3 + 6rs^2 - 7rs + r + s))^2$\\
%$x_{15}=s(r-s)(rs-2r+1)(r^2s^3 - 5r^2s^2 + 6r^2s - r^2 + rs^4 - 3rs^3 + 6rs^2 - 7rs + r + s)$\\
%$\qquad\quad(r^3s^2 - 4r^3s + 2r^3 + 3r^2s^2 + 2r^2s - 2r^2 - rs^5 + 4rs^4 - 10rs^3 + 6rs^2 - 3rs + r + s^4)$\\
%$\qquad\quad / \enspace (r^3 - r^2s^5 + 7r^2s^4 - 18r^2s^3 + 19r^2s^2 - 10r^2s - rs^5 + 4rs^4 - 5rs^2 + 5rs - s^5$\\
%$\qquad\qquad +\enspace s^4 - s^3 + s^2 - s)^2$\\
\bottomrule
\end{tabular}
\vspace{4pt}
\caption{$x$-coordinates of $nP$ for $n\le 10$.}\label{table:Px}
\end{center}
\end{table}

\section{Optimizing plane curve equations}\label{section:optimization}
We now consider how to optimize a given curve equation $F(r,s)=0$ for fast computation.
We have in mind the curves $F_N(r,s)=0$ of the previous section, but our method can be applied more generally.
We seek a birationally equivalent curve $f(x,y)=0$ that minimizes the degree of one of its variables (say~$y$).
Subject to this constraint, we would like to make $f$ monic in~$y$ and to minimize the degree in~$x$, the number of terms,
and the size of the coefficients (roughly in that order of priority).  One technique 
is to attempt to remove singularities through a carefully chosen sequence of translations and
inversions, see \cite{Reichert:NontrivialTorsion} for examples.  Here we take a more na\"{i}ve
approach that allows us to completely automate the process.

There are three basic types of transformations we shall use:
\renewcommand\labelenumi{(\theenumi)}
\begin{enumerate}
\item
Translate: \hspace{18pt}$x\rightsquigarrow x+a$\hspace{48pt} or $\quad y\rightsquigarrow y+a$.
\item
Invert: \hspace{32pt}$x\rightsquigarrow 1/x$\hspace{56pt} or $\quad y \rightsquigarrow 1/y$.
\item
Separate: \hspace{20pt}$x\rightsquigarrow 1/x,\enspace y\rightsquigarrow y/x\quad$ or $\quad x\rightsquigarrow x/y,\enspace y\rightsquigarrow 1/y$.
\end{enumerate}
These are clearly all invertible operations.  The third type combines an inversion
and a division, but we find it works well as an atomic unit.  In order to bound the number of
atomic operations, we let $a\in\{\pm 1\}$, giving a total of eight.

Consider the directed graph $G$ on the set $\mathcal{C}$ of plane curves that can be obtained from $F(r,s)=0$ by applying a finite sequence of the transformations above,
with edges labeled by the corresponding operation.  A path in $G$ defines a birational map (the composition of the operations labeling its edges), and any path can be
reversed to yield the inverse map.  Starting from the given curve $C_0$ defined by $F(r,s)=0$, we want to find a path to a ``better" curve $C$.  To make this precise, we associate to each
integer polynomial $f(x,y)$ a vector of nonnegative integers
$$v(f)=(d_y,m_y,d_x,d_{\rm tot},t,S),$$
whose components are defined by:
\begin{itemize}
\item
$d_x$ is the degree of $f$ in $x$ and $d_y$  is the degree of $f$ in $y$;
\item
$m_y$ is 0 if no term of $f$ is a multiple of $xy^{d_y}$ and 1 otherwise;
\item
$d_{\rm tot}$ is the total degree of $f$;
\item
$t$ is the number of terms in $f$;
\item
$S$ is the sum of the absolute values of the coefficients of $f$.
\end{itemize}
The component $m_y$ will be zero exactly when $f$ can be made monic as a polynomial in $y$.  We order the vectors $v(f)$ lexicographically, and to each $C\in\mathcal{C}$
assign the vector $v(C)=\min\{v(f(x,y)),v(f(y,x))\}$, where $f(x,y)=0$ defines $C$.  We compare curves by comparing their vectors, obtaining a prewellordering of $\mathcal{C}$.
In particular, any subset of $\mathcal{C}$ contains a (not necessarily unique) minimal element.

Given a plane curve $C_0$ we now give a simple algorithm to search the graph $G$ for a birationally equivalent curve $C_1$ that is locally optimal within a radius $R$.
Let $\rho(C,k)$ be the set of curves connected to $C$ by a path of length at most $k$ in $G$, and call $C$ $k$-optimal if $v(C) \le v(C')$ for all $C'\in\rho(C,k)$.
\begin{algorithm}
Given a plane curve $C_0$ defined by an integer polynomial and an integer $R$, output an $R$-optimal curve $C_1$ and a birational map $\varphi:C_1\to C_0$.
\end{algorithm}
\renewcommand\labelenumi{\theenumi.}
\begin{enumerate}
\item
Set $C\leftarrow C_0$, $k=1$, and let $\varphi$ be the identity map.
\item
While $k\le R$:
\renewcommand\labelenumii{\theenumii.}
\begin{enumerate}
\item
Determine a minimal element $C'$ of $\rho(C,k)$.
\item
If $v(C') < v(C)$, then set $\varphi\leftarrow \varphi\circ\phi(C',C)$, $C\leftarrow C'$, and $k\leftarrow 0$.
\item
Set $k\leftarrow k+1$.
\end{enumerate}
\item
Output $C_1=C$ and $\varphi$.
\end{enumerate}

We note that the output curve $C_1$ is birationally equivalent to $C_0$ via the map~$\varphi$, and it is clearly $R$-optimal.
Moreover, $v(C_1)\le v(C)$ for all $C\in\rho(C_0,R)$).

To enumerate the neighbors of the curve $C$ defined by $f(x,y)=0$, the algorithm applies each of the eight atomic operations.  The result of applying the birational map $\phi$ with inverse $\pi$ is computed by expanding $f(\pi_x(x,y),\pi_y(x,y))$ as a formal substitution of variables and clearing any denominators that result.  Thus the translation $x\rightsquigarrow x-1$ is obtained by expanding $f(x+1,y)$, and the inversion $x\rightsquigarrow 1/x$ effectively replaces $x^i$ in $f(x,y)$ with $x^{d_x-i}$.  To enumerate $\rho(C,k)$ we must consider up to $8^k$ possible sequences of operations (this number can be reduced by eliminating obviously redundant sequences), so the bound $R$ cannot be very large.  We have tested up to $R=10$, but find that $R=8$ suffices to obtain the results given here.  With $R=8$ the algorithm takes less than an hour to process the curve $F_N(r,s)=0$ for $N\le 50$ (on a single 2.8 GHz AMD Athlon core).

\begin{example}
Table \ref{table:ExampleSearch} illustrates the algorithm's execution on $F_{16}(r,s)=0$.
We start with $C=C_0$ defined by $f(x,y)=F_{16}(x,y)$.
Since $v(f(y,x)) < v(f(x,y)$, the algorithm determines that
$v(C)=v(f(y,x))=(3,1,5,6,13,40),$
indicating that $f(y,x)$ has degree 3 in $y$, is not monic in $y$, has
degree 5 in $x$, total degree $6$, 13 terms, and coefficients whose absolute values sum to 40.

\begin{table}
\begin{center}
\begin{tabular}{@{}lll@{}}
steps&$f(x,y)$&$v(C)$\\
\toprule
-&$ x^3y^2 - 4x^3y + 2x^3 + 3x^2y^2 + 2x^2y - 2x^2 - xy^5 + 4xy^4$&(3,1,5,6,13,40)\\
&$\quad-\enspace 10xy^3 + 6xy^2 - 3xy + x + y^4$&\\
5,8&$x^3 + x^2y^5 - 3x^2y^4 + 6x^2y^3 - 10x^2y^2 + 4x^2y - x^2 - 2xy^6$&(3,0,7,7,13,40)\\
&$\quad+\enspace 2xy^5 + 3xy^4 + 2y^7 - 4y^6 + y^5$&\\
1,3,8,6&$x^3 + x^2y^4 + 2x^2y^3 + 4x^2y^2 - 5x^2 - 2xy^4 - 8xy^3 - 13xy^2$&(3,0,4,6,13,68)\\
&$\quad+\enspace 8x + 2y^4 + 8y^3 + 10y^2 -4$&\\
1&$x^3 + x^2y^4 + 2x^2y^3 + 4x^2y^2 - 2x^2 - 4xy^3 - 5xy^2 + x$&(3,0,4,6,11,24)\\
&$\quad+\enspace  y^4 + 2y^3 + y^2$&\\
1&$x^3 + x^2y^4 + 2x^2y^3 + 4x^2y^2 + x^2 + 2xy^4 + 3xy^2 + 2y^4$&(3,0,4,6,8,16)\\
5,6&$2x^3 + 3x^2y^2 + 2x^2 + xy^4 + 4xy^2 + 2xy + x + y^4$&(3,0,4,5,8)\\
2,4,5,6,8&$ - x^3 + x^2y^3 - 4x^2y^2 + 4x^2y + 2x^2 + 3xy^2 - 6xy - x + 2y$&(3,0,3,5,9,24)\\
3&$- x^3 + x^2y^3 - x^2y^2 - x^2y + 3x^2 + 3xy^2 - 4x + 2y +2$&(3,3,0,5,9,18)\\
4,5,1,7&$- x^2y^2 - 2x^2y - x^2 + xy^3 + 2xy^2 + y^3 + 3y^2 + 2y$&(2,1,3,4,8,13)\\
4&$- x^2y^2 + xy^3 - xy^2 - xy + x + y^3 - y$&(2,1,3,4,7,7)\\
8&$- x^2 + xy^3 - xy^2 - xy + x - y^3 + y$&(2,0,3,4,7,7)\\
1&$- x^2 + xy^3 - xy^2 - xy - x - y^2$&(2,0,3,4,6,6)\\
\bottomrule
\end{tabular}
\vspace{4pt}
\caption{Optimization of $F_{16}(r,s)=0$.}\label{table:ExampleSearch}
\vspace{-12pt}
$$1: x\rightsquigarrow x-1,\qquad 2: x\rightsquigarrow x+1,\qquad 3: y\rightsquigarrow y-1, \qquad 4: y\rightsquigarrow y+1$$
$$5: x\rightsquigarrow 1/x,\quad\enspace 6: y\rightsquigarrow 1/y,\quad\enspace 7: x\rightsquigarrow 1/x,\enspace y\rightsquigarrow y/x, \quad\enspace 8: x\rightsquigarrow x/y,\enspace y\rightsquigarrow 1/y.$$
\end{center}
\end{table}

No curves within a distance $k=1$ are found that improve $v(C)$, but for $k=2$ a curve $C'$ is found
that is monic in $x$ (and also degree 3), which implies $v(C')<v(C)$.  $C'$ is a minimal curve in $\rho(C,2)$,
so $C$ is replaced by $C'$ and the map $\varphi$ becomes
$$x\rightsquigarrow y/x,\quad y\rightsquigarrow 1/y.$$
This reverses the sequence of steps 5,8 (as identified in the key to Table \ref{table:ExampleSearch}),
used to reach $C'$ from $C_0$ (so $\varphi$ maps points on $C'$ back to points on $C_0$).  The next improvement
occurs when $k=4$.  In this case reversing the path 1,3,8,6 from~$C$ to~$C'$ yields the path 6,8,4,2 from $C'$ back to $C$, and $\varphi$ becomes
$$x\rightsquigarrow (y+1)/(xy+1),\quad y\rightsquigarrow 1/(y+1).$$
The algorithm continues in this fashion, finding the sequence of curves listed in Table \ref{table:ExampleSearch}.
Eventually it is unable to find a better curve within the search radius~$R=8$ and terminates.
The resulting curve $C_1$ has minimal degree in $x$ rather than $y$, so we swap variables (and adjust signs) to obtain the optimized equation 
\begin{equation*}
f_{16}(x,y) = y^2+(x^3+x^2-x+1)y+x^2 = 0,
\end{equation*}
which appears in Table \ref{table:FinalX1}.  Corresponding changes to $\varphi$ yield the birational map
\begin{equation*}
r = 1 + (y + 1) / (xy+y^2),\qquad s= 1+(y+1)/(xy-y^2),
\end{equation*}
listed in Table \ref{table:FinalX1maps}, which carries points on $f_{16}(x,y)=0$ to points on $F_{16}(r,s)=0$.
\end{example}

Table \ref{table:RunStats} shows the improvement in the minimal degree $d(C_i)$ and the number of terms $t(C_i)$ obtained when the initial curve $C_0$ is transformed to the locally optimal curve $C_1$ output by the algorithm.  For comparison, we also list the genus of $X_1(N)$ (sequence $A029937$ in the OEIS \cite{Sloane:OEIS}, or see \cite[Thm.~1.1]{Jeon:ArithmeticModularCurves} for a general formula).

The search procedure described above can, in principal, be applied to any plane curve, but its effectiveness depends largely on finding singularities with small integer coordinates.
Empirically, this works well with the curves $F_N(r,s)=0$, but other applications may wish to modify the list of atomic operations to incorporate more general translations.
More sophisticated local search techniques such as simulated annealing \cite{CadaySutherland:SimulatedAnnealing} can significantly improve performance.

\section{Application to finite fields}\label{section:FiniteFields}

As described in the introduction, we may use our optimized equations for $Y_1(N)$ to efficiently generate elliptic curves containing a point of order~$N$ over a finite field~$\mathbb{F}_q$, where $q$ is prime to $N$.
Here we briefly address a few topics relevant to practical implementation.
We assume we have an optimized equation $f_N(x,y)=0$ for $Y_1(N)$ with $d_y\le d_x$, and consider how we may use $f_N(x,y)$ to efficiently generate a set of~$n$ elliptic curves over $\mathbb{F}_q$, each containing a point of order~$N$.

Except for a small set of points (those leading to singular curves and those for which $\varphi$ is undefined), there is a one-to-one correspondence between $\Fq$-rational points on $f_N(x,y)=0$ and elliptic curves $E/\Fq$ in Tate normal form on which the point $P=(0,0)$ has order~$N$ (see Section \ref{section:rawform}). It follows that we can obtain a uniform distribution of pairs $(E,P)$ from a uniformly distributed sample of the $\Fq$-rational points on $f_N(x,y)=0$.  We should note that this distribution is not uniform on $E$; there is a pair $(E,P)$ for each point $P$ of order $N$ in $E(\Fq)[N]$, and the number of such $P$ may vary with $E$.  The distribution on $E$ can be precisely determined, see~\cite{Gekeler:ECGroupStructures} for the case where $q$ is prime.

When $d_y > 2$ it is not a trivial task to uniformly sample of the zero locus of $f_N$ in an efficient manner.  It is impractical to test random solutions to $f_N(x,y)=0$, so instead we pick $x_i\in\mathbb{F}_q$ at random and compute the $\Fq$-rational roots $y_{ij}$ of the polynomial $h_{i}(y)=f_N(x_i,y)$ (if any).  For each such $y_{ij}$ we include the point $(x_i,y_{ij})$ in our set of $n$ points.  Assuming $n \gg d_y$ this gives us an approximately uniform distribution (if we used only one root of $h_i$ this would not be true), but the points obtained are not all independent.  In practice this does not typically pose a problem.  At most $d$ points share a common $x$ value, and after mapping the points back to $F_N(r,s)=0$ and constructing $E(b,c)$ it is difficult to discern any relationship among the curves.\footnote{We could obtain a uniform independent distribution by using at most one root of each $h_{i}$, provided that we discard it with a certain probability depending on the number of roots $h_{i}$ has, but this is not a very efficient solution.}  With this approach we expect to compute the roots of $n$ polynomials $h_i(y)$, on average, in order to obtain $n$ points on $f_N(x,y)=0$.

When $X_1(N)$ has genus 1, the projective closure of the curve $f_N(x,y)=0$ is an elliptic curve, and we may use a more efficient approach: select a point at random and compute multiples of it via the group operation.  We can generate $n$ random multiples using a total of $O(\log{q}+n\log{q}/\log\log{q})$ group operations via standard multi-exponentiation techniques \cite{Yao:MultiExponentiation}, or we can compute multiples in an arithmetic sequence using just $n+O(\log{q})$ group operations.  The latter approach does not generate independent points, but it is highly efficient: only $O(1)$ operations in $\mathbb{F}_q$ are required per point, assuming $n \gg \log{q}$.  During this computation it is convenient to work with an elliptic curve in short Weierstrass form.  These are provided in Table \ref{table:genus1X1}, along with the corresponding maps back to $F_N(r,s)=0$.

\begin{table}
\begin{center}
\begin{tabular}{@{}rl@{}}
$N$&Weierstrass equation and birational map to $F_N(r,s)=0$\\
\toprule
11&$y^2=x^3-432x+8208$\\
&$r=(y+108)/216$\\
&$s=1+(y-108)/(6x+72)$\\
\midrule
14&$y^2=x^3-675x+13662$\\
&$r=1+(108x - 36y + 3564)/(3x^2-xy-342x+75y+999)$\\
&$s=(6x-234) / (9x-y-135)$\\
\midrule
15&$y^2=x^3-27x+8694$\\
&$t=(6x-90)(18x+6y-918)$\\
&$r=1 - t/(x^2y - 189x^2 + 42xy - 4050x - 3y^2 + 441y - 1701)$\\
&$s=1- t/(x^2y - 81x^2 + 6xy - 3402x - 3y^2 + 981y - 35721)$\\
\bottomrule
\end{tabular}
\vspace{4pt}
\caption{Short Weierstrass models for the genus 1 cases.}\label{table:genus1X1}
\end{center}
\end{table}

Having generated a set of $n$ points on $f_N(x,y)=0$, we apply the appropriate birational map to obtain points on $F_N(r,s)=0$.  When doing so, we invert the denominators in parallel, via the usual Montgomery trick \cite[Alg. 11.15]{Cohen:HECHECC}.  We then compute $(b,c)$ pairs, using $c=s(r-1)$ and $b=rc$.  In a field of characteristic not~$2$ or $3$, we may convert the curve $E(b,c)$ to short Weierstrass form
\begin{equation*}
y^2=x^3+Ax+B.
\end{equation*}
Let $a=c-1$ and $e=a^2-4b$.  We may apply the admissible change of variables
%% AVS 2/6/2011 clarified the substitution and corrected a sign error
\begin{equation}\label{equation:KubertToSWS}
u=(x-3e)/36,\qquad v=(y+108(au+b))/216,
\end{equation}
to the curve $E(b,c)$ defined by $v^2 + (1-c)uv - bv = u^3-bu^2$.  After clearing denominators we obtain the isomorphic curve $y^2=x^3+Ax+B$, where
$$A = 27(24ab-e^2),\qquad\qquad B=54(e^3-36abe+216b^2),$$
%% AVS 7/27/2009 corrected typo, changed 3d to 3e
on which $(3e, -108b)$ is a point of order~$N$.

At some point during the process described above, we need to check that the discriminant $\Delta$ of each curve obtained is nonzero.  This is most efficiently done at the end using $\Delta = -4A^3-27B^2$.  This may result in fewer than $n$ curves being generated, but we can always obtain more points on $f_N(x,y)=0$ (assuming $q\gg m$).

We remark that when $X_1(N)$ has genus 0, parameterizations that additionally provide a point of infinite order over $\mathbb{Q}$ can be found in \cite{Atkin:ECMCurves}.

\section{Prescribing 4-torsion}

For odd $N$, we can use our equations for $Y_1(2N)$ to generate elliptic curves which contain a point of order $4N$ over $\mathbb{F}_q$ in a manner that may be more efficient than using $Y_1(4N)$.
We can also efficiently generate curves which contain a point of order $2N$ but do \emph{not} contain a point of order $4N$.
These results rely on efficiently computing the 4-torsion of an elliptic curve using a known a point of order 2, which we obtain from the point $P=(0,0)$ of order $2N$.
For odd $N$, a curve with a point of order~$N$ has a point of order $4N$ if and only if it has a point of order $4$.

In fact, we only need the $x$-coordinate of $NP$, which can be computed as described in Section \ref{section:rawform}.  It will be convenient to work with the short Weierstrass form, so we assume that the point $NP$ has been translated via (\ref{equation:KubertToSWS}) to the 2-torsion point $\beta=(x_0,0)$ on the curve $E$ defined by $y^2=f(x)=x^3+Ax+B$.

Our strategy is to use the value $x_0$ to determine whether $E(\Fq)$ contains a point of order 4 or not.  In the best case this requires only a single test for quadratic residuacity in $\mathbb{F}_q$, and even in the worst case, a square root and two tests for quadratic residuacity suffice.  If the result is not as desired, we discard $E$ and test another curve with a point of order $2N$.  On average we expect to test two curves.  This is typically faster than using $Y_1(4N)$ or using $Y_1(N)$ and computing 4-torsion without a known point of order 2.

We rely on the following lemma, which is a standard result. Lacking a suitable reference, we give a short proof here.
%We rely on the following lemma.

\begin{lemma}\label{lemma:halving}
If $\alpha=(u,v)$ and $\beta=(x_0,0)$ are points on a nonsingular elliptic curve $E$ defined by $y^2=f(x)=x^3+Ax+B$ over a field of characteristic not $2$ then
$$2\alpha = \beta\qquad \Longleftrightarrow\qquad (u-x_0)^2 = f'(x_0),$$
where $f'(x)=3x^2+A$.
\end{lemma}
\begin{proof}
If $2\alpha=\beta$ then the duplication formula for elliptic curves \cite[p. 59]{Silverman:EllipticCurves1} implies
$$x_0=\frac{u^4-2Au^2-8Bu+A^2}{4(u^3+Au+B)}.$$
Therefore $u$ must satisfy
$$u^4-4x_0u^3-2Au^2-(4Ax_0+8B)u-4Bx_0+A^2 = 0.$$
Since $\beta=(x_0,0)\in E$, we have $x_0^3+Ax_0+B=0$.  Substituting for $B$ yields
$$u^4-4x_0u^3-2Au^2+(8x_0^3+4Ax_0)u+4x_0^4+4Ax_0^2+A^2.$$
We now set $u=z+x_0$ and rewrite this as
$$(z^2-(3x_0^2+A))^2 = 0.$$
Therefore
$$(u-x_0)^2 = 3x_0^2+A = f'(x_0),$$
as desired.  Reversing the argument yields the converse, provided $f(u)\ne 0$. But if $u$ is a root of $f$, then
one can show that $(u-x_0)^2=f'(x_0)$ implies $D(f)=0$, contradicting the fact that $E$ is nonsingular.
\end{proof}

There may be 1 or 3 points of order 2 in $E(\Fq)$.  The $x$-coordinates of the other two (if they exist) are the roots $x_1$ and $x_2$ of $f(x)/(x-x_0)$, which we can determine with the quadratic formula.  We now give our main result for treating 4-torsion.

\renewcommand\labelenumi{(\theenumi)}
\begin{proposition}\label{prop:4torsion}
Let $(x_0,0)$ be a point of order $2$ on a nonsingular elliptic curve $E$ defined by $y^2=f(x)=x^3+Ax+B$ over the field $\mathbb{F}_q$, with quadratic character $\chi$.  Let $n$ be the number of roots of $f(x)$ in $\Fq$, and for $n=3$, let $x_1$ and $x_2$ denote the other two roots.
\vspace{4pt}

For $q\equiv 3\bmod 4$:
\begin{enumerate}
\item
If $\chi(f'(x_0))=1$ then $E(\Fq)$ contains a point of order $4$.
\item
Otherwise, $E(\Fq)$ contains a point of order $4$ if and only if $n=3$ and $\chi(f'(x_1))=1$.
\end{enumerate}

For $q\equiv 1\bmod 4$:
\begin{enumerate}
\item
If $n=1$ then $E(\Fq)$ contains a point of order $4$ if and only if $\chi(f'(x_0))=1$.
\item
Otherwise, if $\chi(f'(x_0))=1$ $($resp., $\chi(f'(x_0))=-1)$ then $E(\Fq)$ contains a point of order~$4$ if and only if $\chi(x_0-x_1)=1$ $($resp., $\chi(x_1-x_2)=1)$.
\end{enumerate}
\end{proposition}
\begin{proof}
Note that $f(x_i)=0$ implies $f'(x_i)\ne 0$, since $E$ is nonsingular, hence $\chi(f'(x_i))=\pm 1$.
Let $\tilde{E}$ denote the quadratic twist of $E$ over $\Fq$.  By Lemma \ref{lemma:halving}, each root $x_i$ of $f(x)$ for which $\chi(f'(x_i))=1$ yields 4 points of order 4 (two pairs of inverses), either all in $E(\Fq)$, all in $\tilde{E}(\Fq)$, or split 2-2 between them.  Recall that $\#E(\Fq)=q+1-t$ and $\#\tilde{E}(\Fq)=q+1+t$, where $t$ is the trace of Frobenius, so $4$ divides $\#E(\Fq)$ if and only if $4$ divides $\#\tilde{E}(\Fq)$, and for $q\equiv 3\bmod 4$, $8$ divides $\#E(\Fq)$ if and only if $8$ divides $\#\tilde{E}(\Fq)$.  

We first consider $q\equiv 3 \bmod 4$.

Suppose $\chi(f'(x_0))=1$.  If $n=1$ then $E(\Fq)$ and $\tilde{E}(\Fq)$ both have two points of order $4$.  If $n=3$ then at least one of $\#E(\Fq)$ or $\tilde{E}(\Fq)$ is divisible by $8$, but then they both are, hence $E(\Fq)$ (and $\tilde{E}(\Fq)$) must contain a point of order $4$, since the $2$-rank of an elliptic curve over $\Fq$ is at most $2$.

Suppose $\chi(f'(x_0))= -1$.  If $n=1$ then $E(\Fq)$ cannot contain a point of order~$4$, so assume $n=3$.  By Lemma \ref{lemma:parity}, for $q\equiv 3\bmod 4$ we have $\chi(f'(x_1))=\chi(f'(x_2))$, and if their common value is $-1$ then $E(\Fq)$ cannot have a point of order 4.  If instead it is $1$, then at least one of $\#E(\Fq)$ or $\#\tilde{E}(\Fq)$ is divisible by $8$, hence both are, and as above, $E(\Fq)$ must contain a point of order $4$.

We now consider $q\equiv 1\bmod 4$.

If $n=1$ then $E(\Fq)$ contains a point of order $4$ if and only if $\chi(f'(x_0))=1$, as above,
so assume $n=3$.  It follows from Theorem 4.2 of~\cite{Knapp:EllipticCurves} that $E(\Fq)$ has a point of order 4 if and only if at least two of $x_0-x_1$, $x_1-x_2$, and $x_2-x_0$ are squares in $\Fq$.  We have $f'(x_0)=(x_0-x_1)(x_0-x_2)$, so if $\chi(f'(x_0))=1$ then it suffices to check $\chi(x_0-x_1)$, and if $\chi(f'(x_0))=-1$ then it suffices to check $\chi(x_1-x_2)$.
\end{proof}

\begin{lemma}\label{lemma:parity}
Let $f(x)$ be a monic cubic polynomial with distinct roots $x_0$,$x_1$,$x_2$ in~$\Fq$, with $q$ odd.  We have
% AVS 03/27/12 corrected two typos below: changed middle x_2 to x_1 and swapped even and odd
$$\chi(-1)\chi(f'(x_0))\chi(f'(x_1))\chi(f'(x_2)) = 1.$$
In particular, the number of squares in the set $\{f'(x_0), f'(x_1), f'(x_2)\}$ is odd when $q\equiv 1 \bmod 4$ and even when $q\equiv 3\bmod 4$.
\end{lemma}
\begin{proof}
Recall that for a monic $f$ of degree $n=3$, the discriminant of $f$ is given by
$$D(f)=(-1)^{n(n-1)/2}R(f,f')=-R(f,f'),$$
where $R(f,f')$ is the resultant.  Since $f$ is monic, we have $R(f,f')\prod f'(x_i)$, thus
$$D(f)=-f'(x_0)f'(x_1)f'(x_2).$$
The roots of $f$ are distinct, so $D(f)\ne 0$.  By the Stickelberger-Swan Theorem (Corollary 1 in \cite{Swan:FactoringPolynomials}), $D(f)$ must be a square in $\mathbb{F}_q$, since $f$ is degree 3 and has 3 irreducible factors.  The lemma then follows, since $\chi(D(f))=1$.
\end{proof}

We note that in Proposition \ref{prop:4torsion}, when (1) fails to hold it is quite likely that $E$ has trivial 4-torsion.
On average, this probability is about 90\% (it can be computed precisely via \cite{Gekeler:ECGroupStructures,Howe:EllipticCurveOrders}).
As a practical optimization, when seeking a point of order $4N$, if condition (1) fails we may simply discard the curve and test another.  When $q\equiv 3 \bmod 4$ this reduces to a test for quadratic residuacity in $\mathbb{F}_q$, and we expect two tests of curves generated with $Y_1(2N)$ will suffice to produce a curve with a point of order $4N$.
\medskip

As a final remark, we note the following generalization to our approach to prescribing 4-torsion.
In \cite{Miret:EC2Sylow}, Miret et al. give an algorithm to determine a point of maximal 2-power order on an elliptic curve over a finite field.
Their algorithm is based on successive halving, and can be accelerated if given as input a point of order $2^i$ for some $i\ge 1$.
Thus we can use their algorithm to efficiently search for an elliptic curve with a point of order $2^kN$ using curves generated with $Y_1(2^iN)$, where $1 \le i < k$.
The optimal choice of $i$ depends on $k$ and the equations for $Y_1(2^iN)$, but will often be $i=1$.
This approach is especially effective in cases where the 2-Sylow subgroup of $E(\Fq)$ is cyclic, which happens more often than not.

\bibliographystyle{amsplain}
%\bibliography{../general}
\input{torsion.bbl}

\section{Appendix}

\noindent
For reasons of space, most of the tables that follow are restricted to $N\le 30$, and we omit some data.
Complete results for $N\le 50$ are available online at \url{http://math.mit.edu/~drew}.

\begin{table}
\begin{center}
\begin{tabular}{@{}ll@{}}
$N$&$F_N(r,s)$\\
\toprule
6&$s-1$\\
7&$r-s$\\
8&$rs-2r+1$\\
9&$r-s^2+s-1$\\
10&$rs^2 - 3rs + r + s^2$\\
11&$r^2-rs^3+3rs^2-4rs+s$\\
12&$r^2s - 3r^2 + rs + 3r - s^2 - 1$\\
13&$r^3 - r^2s^4 + 5r^2s^3 - 9r^2s^2 + 4r^2s - 2r^2 - rs^3 + 6rs^2 - 3rs + r - s^3$\\
14&$r^2s^3 - 5r^2s^2 + 6r^2s - r^2 + rs^4 - 3rs^3 + 6rs^2 - 7rs + r + s$\\
15&$r^3 - r^2s^5 + 7r^2s^4 - 18r^2s^3 + 19r^2s^2 - 10r^2s - rs^5+ 4rs^4 - 5rs^2 + 5rs - s^5$\\
&$+\enspace s^4 - s^3 + s^2 - s$\\
16&$r^3s^2 - 4r^3s + 2r^3 + 3r^2s^2 + 2r^2s - 2r^2 - rs^5 + 4rs^4 - 10rs^3 + 6rs^2 - 3rs$\\
&$+\enspace r + s^4$\\
17&$r^5 - r^4s^6 + 9r^4s^5 - 31r^4s^4 + 50r^4s^3 - 39r^4s^2 + 10r^4s - 3r^4 - r^3s^6 + 3r^3s^5$\\
&$+\enspace 12r^3s^4 - 46r^3s^3 + 54r^3s^2 - 15r^3s + 3r^3 - r^2s^6 - 3r^2s^5+ 9r^2s^4 + r^2s^3 $\\
&$-\enspace21r^2s^2 + 6r^2s - r^2 + rs^7 - 3rs^6 + 6rs^5 - 10rs^4 + 11rs^3 - s^3$\\
18&$r^4s^3 - 6r^4s^2 + 9r^4s - r^4 + r^3s^5 - 7r^3s^4 + 20r^3s^3 - 19r^3s^2 - 8r^3s + r^3 + r^2s^4$\\
&$-\enspace 11r^2s^3 + 28r^2s^2 + rs^4 - 5rs^3 - 8rs^2 + s^4 + s^3 + s^2$\\
19&$r^6 - r^5s^7 + 11r^5s^6 - 48r^5s^5 + 105r^5s^4 - 121r^5s^3 + 69r^5s^2 - 20r^5s - r^5$\\
&$-\enspace 2r^4s^7 + 12r^4s^6 - 9r^4s^5 - 60r^4s^4 + 144r^4s^3 - 105r^4s^2 + 35r^4s - 3r^3s^7$\\
&$+\enspace 3r^3s^6 + 21r^3s^5 - 30r^3s^4 - 41r^3s^3 + 51r^3s^2 - 21r^3s + r^2s^9 - 6r^2s^8 + 21r^2s^7$\\
&$-\enspace 50r^2s^6+ 66r^2s^5 - 31r^2s^4 + 25r^2s^3 - 18r^2s^2 + 7r^2s + 3rs^6 - 15rs^5 + 10rs^4$\\
&$-\enspace 6rs^3+ 3rs^2 - rs + s^6$\\
20&$r^5s^2 - 5r^5s + 5r^5 + 5r^4s^2 - 10r^4 - r^3s^7 + 9r^3s^6 - 35r^3s^5 + 70r^3s^4 - 85r^3s^3$\\
&$+\enspace 51r^3s^2 - 9r^3s + 10r^3 + 10r^2s^5 - 35r^2s^4 + 60r^2s^3 - 50r^2s^2 + 10r^2s - 5r^2$\\
&$-\enspace rs^7+ 3rs^6 - 6rs^5 + 10rs^4 - 15rs^3 + 16rs^2 - 3rs + r - s^2$\\
21&$r^6 - r^5s^8 + 13r^5s^7 - 69r^5s^6 + 192r^5s^5 - 300r^5s^4 + 261r^5s^3 - 119r^5s^2$\\
&$+\enspace 21r^5s - 4r^5 - r^4s^9 + 10r^4s^8 - 45r^4s^7 + 141r^4s^6 - 345r^4s^5 + 576r^4s^4$\\
&$-\enspace 551r^4s^3 + 273r^4s^2 - 49r^4s + 6r^4 - r^3s^10 + 10r^3s^9 - 51r^3s^8 + 159r^3s^7$\\
&$-\enspace 316r^3s^6 + 450r^3s^5 - 551r^3s^4 + 489r^3s^3 - 247r^3s^2 + 42r^3s - 4r^3 + 3r^2s^8$\\
&$-\enspace 31r^2s^7 + 109r^2s^6 - 172r^2s^5 + 203r^2s^4 - 181r^2s^3 + 97r^2s^2 - 14r^2s + r^2$\\
&$+\enspace 2rs^8 - 11rs^7 + 8rs^6 + 2rs^5 - 13rs^4 + 19rs^3 - 14rs^2 + rs + s^8 - s^7 + s^6 - s^5$\\
&$+\enspace s^4 - s^3 + s^2$\\
22&$r^6s^5 - 9r^6s^4 + 28r^6s^3 - 35r^6s^2 + 15r^6s - r^6 + r^5s^8 - 12r^5s^7 + 59r^5s^6$\\
&$-\enspace 148r^5s^5 + 205r^5s^4 - 186r^5s^3 + 133r^5s^2 - 49r^5s + 3r^5 + r^4s^8 - 6r^4s^7$\\
&$-\enspace 8r^4s^6 + 118r^4s^5 - 260r^4s^4 + 249r^4s^3 - 164r^4s^2 + 58r^4s - 3r^4 + r^3s^8$\\
&$-\enspace 30r^3s^6 + 34r^3s^5 + 70r^3s^4 - 106r^3s^3 + 80r^3s^2 - 30r^3s + r^3 + r^2s^8$\\
&$+\enspace 6r^2s^7 - 7r^2s^6 - 25r^2s^5 + 5r^2s^4 + 14r^2s^3 - 16r^2s^2 + 7r^2s - rs^9 + 3rs^8$\\
&$-\enspace 8rs^7 + 21rs^6 - 15rs^5 + 10rs^4 - 6rs^3 + 3rs^2 - rs - s^7$\\
23&$r^9 - r^8s^9 + 15r^8s^8 - 94r^8s^7 + 319r^8s^6 - 636r^8s^5 + 756r^8s^4 - 520r^8s^3$\\
&$+\enspace 189r^8s^2 - 35r^8s - 2r^8 - 4r^7s^9 + 39r^7s^8 - 120r^7s^7 + 28r^7s^6 + 597r^7s^5$\\
&$-\enspace 1341r^7s^4 + 1256r^7s^3 - 525r^7s^2 + 105r^7s + r^7 - 10r^6s^9 + 45r^6s^8 + 24r^6s^7$\\
&$-\enspace 357r^6s^6 + 324r^6s^5 + 570r^6s^4 - 1130r^6s^3 + 576r^6s^2 - 126r^6s + r^5s^{13} - 14r^5s^{12}$\\
&$+\enspace 93r^5s^{11} - 370r^5s^{10} + 970r^5s^9 - 1827r^5s^8 + 2553r^5s^7 - 2296r^5s^6 + 1095r^5s^5$\\
&$-\enspace 480r^5s^4 + 686r^5s^3 - 369r^5s^2 + 84r^5s + r^4s^{12} - 21r^4s^{11} + 165r^4s^{10} - 650r^4s^9$\\
&$+\enspace 1530r^4s^8 - 2562r^4s^7 + 2957r^4s^6 - 2046r^4s^5 + 780r^4s^4 - 415r^4s^3 + 171r^4s^2$\\
&$-\enspace 36r^4s + r^3s^{12} - 15r^3s^{11} + 66r^3s^{10} - 84r^3s^9 - 45r^3s^8 + 402r^3s^7$\\
&$-\enspace 833r^3s^6 + 837r^3s^5 - 351r^3s^4 + 145r^3s^3 - 48r^3s^2 + 9r^3s + r^2s^{12} - 9r^2s^{11}$\\
&$+\enspace 13r^2s^{10} - r^2s^9 - 24r^2s^8 + 28r^2s^7 + 42r^2s^6 - 126r^2s^5 + 56r^2s^4 - 21r^2s^3$\\
&$+\enspace 6r^2s^2 - r^2s + rs^{12} - 3rs^{11} + 6rs^{10} -10rs^9 + 15rs^8 - 21rs^7 + 21rs^6 - s^6$\\
\bottomrule
\end{tabular}
\vspace{4pt}
\caption{Raw equations $F_N(r,s)=0$.}\label{table:RawX1}
\end{center}
\end{table}

\begin{table}
\begin{center}
\begin{tabular}{@{}rrrrrrrr@{}}
$N$&$\enspace g$&$\qquad d(C_0)$&$d(C_1)$&$\qquad t(C_0)$&$t(C_1)$&$\qquad k_{\max}$&$\ell(C_0,C_1)$\\
\toprule
10&0&4&\bf{0}&1&1&2&10\\
11&1&2&\bf{2}&5&4&2&4\\
12&0&2&\bf{0}&6&1&2&13\\
13&2&3&\bf{2}&11&6&2&13\\
14&1&2&\bf{2}&10&4&2&11\\
15&1&3&\bf{2}&15&5&3&18\\
16&2&3&\bf{2}&13&6&5&23\\
17&5&5&\bf{4}&28&12&5&23\\
18&2&4&\bf{2}&19&6&5&24\\
19&7&6&\bf{5}&39&18&4&23\\
20&3&5&\bf{3}&28&6&4&23\\
21&5&6&\bf{4}&55&11&4&18\\
22&6&6&\bf{4}&50&17&7&40\\
23&12&9&\bf{7}&87&38&7&25\\
24&5&6&\bf{5}&41&20&6&25\\
25&12&10&\bf{8}&114&46&6&20\\
26&10&8&\bf{7}&82&27&5&32\\
27&13&11&\bf{8}&135&52&4&19\\
28&10&10&\bf{7}&115&30&2&16\\
29&22&14&\bf{11}&214&88&8&32\\
30&9&10&\bf{8}&109&46&7&23\\
31&26&16&\bf{13}&279&124&6&23\\
32&17&13&\bf{10}&190&78&7&19\\
33&21&16&\bf{12}&319&109&6&29\\
34&21&14&\bf{11}&235&88&7&22\\
35&25&19&\bf{15}&438&142&4&19\\
36&17&14&\bf{11}&224&94&7&23\\
37&40&23&\bf{18}&582&225&4&19\\
38&28&18&\bf{14}&383&140&6&27\\
39&33&22&\bf{17}&586&212&4&20\\
40&25&19&\bf{15}&412&171&5&22\\
41&51&28&\bf{22}&870&336&8&49\\
42&25&20&\bf{15}&442&165&8&27\\
43&57&31&\bf{24}&1065&408&6&23\\
44&36&24&\bf{19}&654&208&3&21\\
45&41&29&\bf{23}&960&368&4&19\\
46&45&26&\bf{21}&791&285&6&23\\
47&70&37&\bf{29}&1526&1768&6&33\\
48&37&26&\bf{19}&773&257&7&23\\
49&69&39&\bf{31}&1791&900&6&37\\
50&48&30&\bf{23}&1040&391&8&42\\
\bottomrule
\end{tabular}
\vspace{8pt}
\caption{Optimization results of Algorithm~2.}\label{table:RunStats}
\end{center}
\begin{minipage}{1.0\linewidth}
\small
The curves $C_0$ and $C_1$ are the raw and optimized forms of $Y_1(N)$, respectively, and $g$ is the genus of $X_1(N)$.  The column $d(C_i)$ lists the minimum of the degree of $C_i$ in~$x$ or~$y$, and $t(C_i)$ is the number of terms.
The column $\ell(C_0,C_1)$ gives the length of the path traveled by the algorithm of Section \ref{section:optimization} to reach $C_1$ from $C_0$ (typically not a shortest path), and $k_{\max}\le R=8$ is the maximum search radius used prior to reaching $C_1$.
\end{minipage}
\end{table}

\begin{table}
\begin{center}
\begin{tabular}{@{}ll@{}}
$N$&$f_N(x,y)$\\
\toprule
11&$y^2 + (x^2 + 1)y + x$\\
13&$y^2 + (x^3 + x^2 + 1)y - x^2 - x$\\
14&$y^2 + (x^2 + x)y + x$\\
15&$y^2 + (x^2 + x + 1)y + x^2$\\
16&$y^2 + (x^3 + x^2 - x + 1)y + x^2$\\
17&$y^4 + (x^3 + x^2 - x + 2)y^3 + (x^3 - 3x + 1)y^2 - (x^4 + 2x)y + x^3 + x^2$\\
18&$y^2 + (x^3 - 2x^2 + 3x + 1)y + 2x$\\
19&$y^5 - (x^2 + 2)y^4 - (2x^3 + 2x^2 + 2x - 1)y^3 + (x^5 + 3x^4 + 7x^3 + 6x^2 + 2x)y^2$\\
&\hspace{12pt}$- (x^5 + 2x^4 + 4x^3 + 3x^2)y + x^3 + x^2$\\
20&$y^3 + (x^2 + 3)y^2 + (x^3 + 4)y + 2$\\
21&$y^4 + (3x^2 + 1)y^3 + (x^5 + x^4 + 2x^2 + 2x)y^2 + (2x^4 + x^3 + x)y + x^3$\\
22&$y^4 + (x^3 + 2x^2 + x + 2)y^3 + (x^5 + x^4 + 2x^3 + 2x^2 + 1)y^2$\\
&\hspace{12pt}$+ (x^5 - x^4 - 2x^3 - x^2 - x)y - x^4 - x^3$\\
23&$y^7 + (x^5 - x^4 + x^3 + 4x^2 + 3)y^6 + (x^7 + 3x^5 + x^4 + 5x^3 + 7x^2 - 4x + 3)y^5$\\
&\hspace{12pt}$+ (2x^7 + 3x^5 - x^4 - 2x^3 - x^2 - 8x + 1)y^4$\\
&\hspace{12pt}$+ (x^7 - 4x^6 - 5x^5 - 6x^4 - 6x^3 - 2x^2 - 3x)y^3$\\
&\hspace{12pt}$- (3x^6 - 5x^4 - 3x^3 - 3x^2 - 2x)y^2 + (3x^5 + 4x^4 + x)y - x^2(x+1)^2$\\
24&$y^5 + (x^4 + 4x^3 + 3x^2 - x - 2)y^4 - (2x^4 + 8x^3 + 7x^2 - 1)y^3$\\
&\hspace{12pt}$- (2x^5 + 4x^4 - 3x^3 - 5x^2 - x)y^2 + (2x^5 + 5x^4 + 2x^3)y + x^6 + x^5$\\
25&$y^8 + (4x^2 + 7x - 4)y^7 - (x^5 - x^4 - 14x^3 - 4x^2 + 24x - 6)y^6$\\
&\hspace{12pt}$- (x^7 + 4x^6 - 3x^5 - 18x^4 + 15x^3 + 33x^2 - 30x + 4)y^5$\\
&\hspace{12pt}$- (x^8 + 2x^7 - 8x^6 - 14x^5 + 24x^4 + 17x^3 - 41x^2 + 16x - 1)y^4$\\
&\hspace{12pt}$+ (x^8 + 6x^7 + 3x^6 - 20x^5 - 3x^4 + 28x^3 - 19x^2 + 3x)y^3$\\
&\hspace{12pt}$- (3x^7 + 9x^6 - 3x^5 - 13x^4 + 11x^3 - 3x^2)y^2 + (3x^6 + 4x^5 - 4x^4 + x^3)y - x^5$\\
26&$y^6 + (3x^2 + 4x - 2)y^5 + (3x^4 + 10x^3 - 9x + 1)y^4$\\
&\hspace{12pt}$+ (x^6 + 7x^5 + 8x^4 - 14x^3 - 11x^2 + 6x)y^3$\\
&\hspace{12pt}$+ (x^7 + 4x^6 - x^5 - 13x^4 + 2x^3 + 10x^2 - x)y^2$\\
&\hspace{12pt}$- (x^6 - 7x^4 - 4x^3 + 2x^2)y - x^4 - x^3$\\
27&$y^8 + (3x^2 + 6x - 3)y^7 - (3x^5 - 18x^3 - 9x^2 + 18x - 3)y^6$\\
&\hspace{12pt}$- (x^8 + 8x^7 + 13x^6 - 21x^5 - 48x^4 + 20x^3 + 42x^2 - 18x + 1)y^5$\\
&\hspace{12pt}$- (x^{10} + 6x^9 + 12x^8 - 14x^7 - 72x^6 - 27x^5 + 93x^4 + 33x^3 - 45x^2 + 6x)y^4$\\
&\hspace{12pt}$+ (x^{10} + 11x^9 + 40x^8 + 36x^7 - 69x^6 - 105x^5 + 33x^4 + 54x^3 - 15x^2)y^3$\\
&\hspace{12pt}$- (4x^9 + 30x^8 + 63x^7 + 10x^6 - 69x^5 - 24x^4 + 19x^3)y^2$\\
&\hspace{12pt}$+ (6x^8 + 27x^7 + 27x^6 - 6x^5 - 12x^4)y - 3x^7 - 6x^6 - 3x^5$\\
28&$y^7 + 3xy^6 + (x^5 + 3x^4 + 5x^3 + 9x^2 + 2x)y^5 - (2x^5 - 6x^3 + 2x^2 + 2x)y^4$\\
&\hspace{12pt}$+ (3x^6 + 16x^5 + 18x^4 - 2x^2)y^3 + (x^7 - 2x^6 - 20x^5 - 28x^4 - 12x^3 - 2x^2)y^2$\\
&\hspace{12pt}$- (2x^7 + 3x^6 - 5x^5 - 10x^4 - 5x^3 - x^2)y + x^7 + 2x^6 + x^5$\\
29&$y^{11} + (2x^3 + 5x^2 + 5x - 3)y^{10}+(x^6 + 8x^5 + 18x^4 + 11x^3 - 5x^2 - 12x + \cdots$\\
%&\hspace{12pt}$+ (3x^8 + 15x^7 + 29x^6 + 6x^5 - 39x^4 - 19x^3 + 5x^2 + 5x - 1)y^8$\\
%&\hspace{12pt}$+ (3x^{10} + 14x^9 + 18x^8 - 26x^7 - 99x^6 - 45x^5 + 95x^4 + 25x^3 - 37x^2 + 7x)y^7$\\
%&\hspace{12pt}$+ (x^{12} + 5x^{11} - 44x^9 - 106x^8 - 40x^7 + 197x^6+ 190x^5 - 140x^4 - 93x^3 + 59x^2 - 6x)y^6$\\
%\&\hspace{12pt}$- (2x^{12} + 16x^{11} + 37x^{10} - 9x^9 - 184x^8 - 256x^7 + 99x^6 + 346x^5 - 20x^4 - 130x^3 + 32x^2 - x)y^5$\\
%&\hspace{12pt}$+ (x^{12} + 15x^{11} + 65x^{10} + 99x^9 - 55x^8 - 320x^7 - 165x^6+ 223x^5 + 100x^4 - 66x^3 + 5x^2)y^4$\\
%&\hspace{12pt}$- (4x^{11} + 36x^{10} + 108x^9 + 98x^8 - 110x^7 - 191x^6 + 15x^5 + 64x^4 - 10x^3)y^3$\\
%&\hspace{12pt}$+ (6x^{10} + 38x^9 + 76x^8 + 25x^7 - 55x^6 - 26x^5 + 10x^4)y^2$\\
%&\hspace{12pt}$- (4x^9 + 17x^8 + 18x^7 - 5x^5)y + x^6(x+1)^2$\\
30&$y^8 - (2x^3 + 4x^2 + x + 5)y^7 + (x^6 + 4x^5 + 6x^4 + 9x^3 + 14x^2 + 10)y^6$\\
&\hspace{12pt}$- (x^7 + 4x^6 + 9x^5 + 10x^4 + 4x^3 + 15x^2 - 10x + 10)y^5$\\
&\hspace{12pt}$+ (x^8 + 4x^7 + 4x^6 - 5x^4 - 20x^3 + 5x^2 - 20x + 5)y^4$\\
&\hspace{12pt}$+ (3x^7 + 11x^6 + 15x^5 + 9x^4 + 18x^3 - 9x^2 + 14x - 1)y^3$\\
&\hspace{12pt}$+ (3x^6 + 9x^5 + 14x^4 + 2x^3 + 13x^2 - 3x)y^2 + (x^5 + x^4 + 4x^3 - 3x^2)y - x^3$\\
\bottomrule
\end{tabular}
\vspace{4pt}
\caption{Optimized equations $f_N(x,y)=0$.}\label{table:FinalX1}
\end{center}
\begin{minipage}{1.0\linewidth}
\small
$f_N(x,y)=x$ for $N\in\{6,7,8,9,10,12\}$.  The polynomial $f_{29}$ is not displayed in full.
Complete polynomials for $N\le 50$ are available at \url{http://math.mit.edu/~drew}.
\end{minipage}
\normalsize
\end{table}

\begin{table}
\begin{center}
\begin{tabular}{@{}rll@{}}
$N$&$\varphi$\\
\toprule
6&$r=x,$\hspace{91pt}$s=1$\\
7&$r=x,$\hspace{91pt}$s=x$\\
8&$r=1/(2-x),$\hspace{56pt}$s=x$\\
9&$r=x^2-x+1,$\hspace{51pt}$s=x$\\
10&$r=-x^2/(x^2-3x+1),$\hspace{15pt}$s=x$\\
11&$r=1+xy,$\hspace{68pt}$s=1-x$\\
12&$r=(2x^2-2x+1) / x,$\hspace{22pt}$s=(3x^2-3x+1) / x^2$\\
13&$r=1-xy,$\hspace{67pt}$s=1-xy/(y + 1)$\\
14&$r=1-(x+y)/((y+1)(x+y+1)),$\hspace{42pt}$s=(1-x)/(y+1)$\\
15&$r=1+(xy+y^2) / (x^3+x^2y+x^2),$\hspace{50pt}$s=1+y/(x^2+x)$\\
16&$r=(x^2 - xy + y^2 + y)/(x^2 + x - y - 1),$\hspace{25pt}$s=(x - y)/(x + 1)$\\
17&$r=(x^2 + x - y)/(x^2 + xy + x - y^2 - y),$\hspace{24pt}$ s=(x + 1)/(x + y + 1)$\\
18&$r=(x^2 - xy - 3x + 1)/((x-1)^2(xy+1)),$\\
     &$s=x^2 - 2x - y)/(x^2 - xy - 3x - y^2 - 2y)$\\
19&$r=1+x(x+y)(y-1)/((x+1)(x^2-xy+2x-y^2+y)),$\\
    &$s = 1+x(y-1)/((x+1)(x-y+1))$\\
20&$r=1+(x^3+xy+x)/((x-1)^2(x^2-x+y+1)),$\\
    &$s=1+(x^2+y+1)/((x-1)(x^2-x+y+2))$\\
21&$r=1+(y^2+y)(xy+y+1)/((xy+1)(xy-y^2+1)),$\\
     &$s=1+(y^2+y)/(xy+1)$\\
22&$r=(x^2y + x^2 + xy + y)/(x^3 + 2x^2 + y),$\hspace{26pt}$s=(xy + y)/(x^2 + y)$\\
23&$r=(x^2 + x + y + 1)/(x^2 - xy),$\hspace{63pt}$s=(x + y + 1)/x$\\
24&$r=(x^2 + x - y + 1)/(x^2 + xy - y^2 + y),$\hspace{23pt}$s=(x + 1)/(x + y)$\\
25&$r=(x^2 + xy + y^2 - y)/(x^2 + x + y - 1),$\hspace{14pt}$s=(x + y)/(x + 1)$\\
26&$r=(x^3y + 3x^2y - x^2 + xy^2) /((x+1)(x^2y+x^2+3xy+y^2)),$\\
   &$s=(xy - x)/(xy + y)$\\
27&$r=(-x^3 - x^2 - x - y)/(x^2y + xy - x - y),$\hspace{9pt}$s=(-x^2 - x - y)/(xy - x - y)$\\
28&$r=1+(xy+y) / ((y-1)(xy-x+2y-1)),$\\
    &$s=1-(xy+y) / ((y-1)(x-y+1))$\\
29&$r=(-x^3 - x^2 - x - y)/(x^2y + xy - x - y)$\\
    &$s=1-(x^2 + xy)/(xy - x - y)$\\
30&$r=(x^2y + x + y)/(x^2y - xy + x)$,\\
&$s=(x^2y + xy + x + y)/(x^2y + x)$\\
\bottomrule
\end{tabular}
\vspace{4pt}
\caption{Birational maps from $f_N(x,y)=0$ to $F_N(r,s)=0$.}\label{table:FinalX1maps}
\end{center}
\end{table}

\end{document}

%% file: torsion.bbl
\providecommand{\bysame}{\leavevmode\hbox to3em{\hrulefill}\thinspace}
\providecommand{\MR}{\relax\ifhmode\unskip\space\fi MR }
\providecommand{\MRhref}[2]{%
  \href{http://www.ams.org/mathscinet-getitem?mr=#1}{#2}
}
\providecommand{\href}[2]{#2}